\documentclass[12pt]{article}

\usepackage{amsmath, epsfig, cite, setspace}
\usepackage{amssymb}
\usepackage{amsfonts}
\usepackage{latexsym}
\usepackage{amsthm}

\def\l{\left}
\def\r{\right}
\def\bg{\bigg}
\def\({\bg(}
\def\){\bg)}
\def\t{\text}
\def\f{\frac}

\def\ls{\leqslant}

\def\eq{\equiv}

\def\Z{\mathbb Z}
\def\C{\mathbb C}
\def\N{\mathbb N}

\def\<{\langle}
\def\>{\rangle}

\makeatletter
\renewcommand{\@seccntformat}[1]{{\csname the#1\endcsname}.\hspace{.5em}}
\makeatother

\newtheorem{theorem}{Theorem}[section]

\newtheorem{conjecture}[theorem]{Conjecture}
\newtheorem{lemma}{Lemma}[section]

\renewcommand{\qed}{\hfill$\Box$\medskip}
\renewcommand{\thefootnote}{*}

\numberwithin{equation}{section}

\begin{document}

\begin{center}
{\large\bf  Refinements of Van Hamme's (E.2) and (F.2) supercongruences\\[5pt] and two supercongruences by Swisher}
\end{center}

\vskip 2mm \centerline{Victor J. W. Guo$^1$ and Chen Wang$^2$\footnote{Corresponding author.} }
\begin{center}
{\footnotesize $^1$School of Mathematics, Hangzhou Normal University, Hangzhou 311121, People's Republic of China\\
{\tt  jwguo@math.ecnu.edu.cn}  \\
$^2$Department of Applied Mathematics, Nanjing Forestry University, Nanjing 210037,  People's Republic of China\\
{\tt  cwang@smail.nju.edu.cn} }
\end{center}

\vskip 0.7cm \noindent{\bf Abstract.} In 1997, Van Hamme proposed 13 supercongruences on truncated hypergeometric series.
Van Hamme's (B.2) supercongruence was first confirmed by Mortenson and received a WZ proof by Zudilin later.
In 2012, using the WZ method again, Sun extended Van Hamme's (B.2) supercongruence to the modulus $p^4$ case, where $p$
is an odd prime. In this paper, by using a more general WZ pair, we generalize Hamme's (E.2) and (F.2) supercongruences,
as well as two supercongruences by Swisher, to the modulus $p^4$ case. Our generalizations of these supercongruences
are related to Euler polynomials. We also put forward a relevant conjecture on $q$-congruences for further study.

\vskip 3mm \noindent {\it Keywords}: supercongruence; Legendre symbol; Euler number; Euler polynomials; $q$-congruence
\vskip 0.2cm \noindent{\it AMS Subject Classifications}:  33C20; 11B75; 11B65; 33E50

\renewcommand{\thefootnote}{**}

\section{Introduction}

In 1914, Ramanujan \cite{Ramanujan} listed a number of fastly convergent series for $1/\pi$.  Although Bauer's formula \cite{Bauer} is not listed in \cite{Ramanujan}, it gives such an example:
\begin{align}
\sum^{\infty}_{k=0}(-1)^k(4k+1)\frac{(\frac{1}{2})_k^3}{k!^3} =\frac{2}{\pi}, \label{eq:Ramanujan}
\end{align}
where $(a)_k=a(a+1)\cdots (a+k-1)$ stands for the Pochhammer symbol.
Ramanujan's formulas for $1/\pi$ became famous in 1980's
when they were discovered to provide efficient algorithms for
evaluating decimal digits of $\pi$ (see \cite{BB}).

In 1997, Van Hamme \cite{Hamme} observed that 13 Ramanujan-type series have neat $p$-adic analogues, such as
\begin{align}
\sum^{(p-1)/2}_{k=0}(-1)^k(4k+1)\frac{(\frac{1}{2})_k^3}{k!^3} &\equiv p(-1)^{(p-1)/2} \pmod{p^3},\label{eq:b2} \\
\sum_{k=0}^{(p-1)/3} (-1)^k(6k+1)\frac{(\frac{1}{3})_k^3}{k!^3} &\equiv p\pmod{p^3}\quad\text{for}\ p\equiv 1\pmod{3},  \label{eq:e2}\\[5pt]
\sum_{k=0}^{(p-1)/4}(-1)^k(8k+1)\frac{(\frac{1}{4})_k^3}{k!^3} &\equiv p(-1)^{(p-1)/4}\pmod{p^3} \ \text{for}\ p\equiv 1\pmod{4}, \label{eq:f2}
\end{align}
(tagged (B.2), (E.2), and (F.2) in his list, respectively), where $p$ is an odd prime. Note that we may compute the sum in \eqref{eq:b2} for $k$ up to $p-1$,
since the $p$-adic order of $(\frac{1}{2})_k/k!$ is $1$ for $k$ satisfying $(p+1)/2\leqslant k\leqslant p-1$.
Congruences of this kind are known as Ramanujan-type supercongruences nowadays. The supercongruence \eqref{eq:b2} was first confirmed by Mortenson \cite{Mortenson}
employing a $_6F_5$ transformation and the $p$-adic Gamma function in 2008, and reproved by Zudilin \cite{Zudilin} with the help of the WZ (Wilf--Zeilberger \cite{WZ1,WZ2}) method, and
by Long \cite{Long} using hypergeometric identities. Swisher \cite{Swisher} utilized Long's method to prove \eqref{eq:e2} and \eqref{eq:f2}. Almost at the same time,
He \cite{He} also applied Long's method to present a generalization of \eqref{eq:e2} and \eqref{eq:f2}.

In 2012, making use of the WZ method again, Sun \cite{Sun} proved the following refinement of \eqref{eq:b2}: for any prime $p>3$,
\begin{align}
\sum^{m}_{k=0}(-1)^k(4k+1)\frac{(\frac{1}{2})_k^3}{k!^3} \equiv p(-1)^{(p-1)/2}+p^3E_{p-3} \pmod{p^4},\label{eq:Sun}
\end{align}
where $m=p-1$ or $(p-1)/2$, and $E_{p-3}$ is the $(p-3)$th Euler number, which may be defined as
$$
\sum_{n=0}^\infty E_n\frac{x^n}{n!}=\frac{2}{e^x+e^{-x}}.
$$
Recall that the Euler polynomials $E_n(x)$ can be defined by the following generating function:
$$
\sum_{n=0}^\infty E_n(x)\frac{t^n}{n!}=\frac{2e^{xt}}{e^t+1}.
$$
It is easy to see that $E_{p-3}=2^{p-3}E_{p-3}(\frac12)\equiv \frac{1}{4}E_{p-3}(\frac12)\pmod{p}$ for any prime $p>3$.

In this paper, we shall establish the following refinements of \eqref{eq:e2} and \eqref{eq:f2}.

\begin{theorem}\label{thm:main-1}
Let $p\equiv 1\pmod{3}$ be a prime. Then
\begin{align}
\sum_{k=0}^{M}(-1)^k(6k+1)\frac{(\frac{1}{3})_k^3}{k!^3}
\equiv p+\frac{p^3}{9}E_{p-3}(\tfrac{1}{3}) \pmod{p^4}, \label{eq:e2-mod4}
\end{align}
where $M=(p-1)/3$ or $p-1$.
\end{theorem}

\begin{theorem}\label{thm:main-2}
Let $p\equiv 1\pmod{4}$ be a prime. Then
\begin{align}
\sum_{k=0}^{M}(-1)^k(8k+1)\frac{(\frac{1}{4})_k^3}{k!^3}
\equiv p(-1)^{(p-1)/4}+\frac{p^3}{16}E_{p-3}(\tfrac{1}{4}) \pmod{p^4},  \label{eq:f2-mod4}
\end{align}
where $M=(p-1)/4$ or $p-1$.
\end{theorem}

In 2015, Swisher \cite[(1), (2)]{Swisher} proved the following supercongruences, which are very similar to \eqref{eq:e2} and \eqref{eq:f2}:
\begin{align}
\sum_{k=0}^{(2p-1)/3} (-1)^k(6k+1)\frac{(\frac{1}{3})_k^3}{k!^3} &\equiv -2p\pmod{p^3}\quad\text{for}\ p\equiv 2\pmod{3},  \label{eq:swisher-e2}\\[5pt]
\sum_{k=0}^{(3p-1)/4}(-1)^k(8k+1)\frac{(\frac{1}{4})_k^3}{k!^3} &\equiv 3p(-1)^{(3p-1)/4}\pmod{p^3} \ \text{for}\ p\equiv 3\pmod{4}, \label{eq:swisher-f2}
\end{align}
where $p$ is a prime.

In this paper, we shall give the following refinements of \eqref{eq:swisher-e2} and \eqref{eq:swisher-f2}.

\begin{theorem}\label{thm:main-3}
Let $p\equiv 2\pmod{3}$ be a prime. Then
\begin{align}
\sum_{k=0}^{M}(-1)^k(6k+1)\frac{(\frac{1}{3})_k^3}{k!^3}
\equiv -2p+\frac{8p^3}{9}E_{p-3}(\tfrac{1}{3}) \pmod{p^4},  \label{eq:sw-e2-mod4}
\end{align}
where $M=(2p-1)/3$ or $p-1$.
\end{theorem}

\begin{theorem}\label{thm:main-4}
Let $p\equiv 3\pmod{4}$ be a prime. Then
\begin{align}
\sum_{k=0}^{M}(-1)^k(8k+1)\frac{(\frac{1}{4})_k^3}{k!^3}
\equiv 3p(-1)^{(3p-1)/4}+\frac{27p^3}{16}E_{p-3}(\tfrac{1}{4}) \pmod{p^4},  \label{eq:sw-f2-mod4}
\end{align}
where $M=(3p-1)/4$ or $p-1$.
\end{theorem}

For any $p$-adic integer $x$, let $\langle x\rangle_p$ stands for the least non-negative residue of $x$ modulo $p$.
We shall prove Theorems \ref{thm:main-1}--\ref{thm:main-4} by establishing the following more general result.

\begin{theorem}\label{main1}
Let $p>3$ be a prime and let $\alpha$ be a $p$-adic integer. Then
\begin{equation}\label{main1eq}
\sum_{k=0}^{p-1}(-1)^k(2k+\alpha)\f{(\alpha)_k^3}{(1)_k^3}\eq (-1)^{\langle -\alpha\rangle_p}(\alpha+\langle -\alpha\rangle_p)+(\alpha+\langle -\alpha\rangle_p)^3E_{p-3}(\alpha)\pmod{p^4},
\end{equation}
and
\begin{equation}\label{main1eq'}
\sum_{k=0}^{\langle -\alpha\rangle_p}(-1)^k(2k+\alpha)\f{(\alpha)_k^3}{(1)_k^3}\eq (-1)^{\langle -\alpha\rangle_p}(\alpha+\langle -\alpha\rangle_p)+(\alpha+\langle -\alpha\rangle_p)^3E_{p-3}(\alpha)\pmod{p^4}.
\end{equation}
\end{theorem}

Note that the case $\alpha=-1/2$ of \eqref{main1eq'} was proved by the first author and Liu \cite[Theorem 1.1]{GuoLiu}. The modulus $p^3$ case of Theorem \ref{main1} was obtained by the second author and Sun \cite[Theorem 1.3]{WangSun}.
Moreover, $q$-analogues of this result modulo $p^3$ for $\alpha=r/d$ being a rational $p$-adic integer and $p\equiv \pm r\pmod{d}$
were given by the first author \cite[Theorems 1.5 and 1.6]{Guo} and the first author and Zudilin \cite[Theorems 4.9 and 4.10]{GuoZu}.

The paper is organized as follows. In the next section, we give some preliminary results. The proof of Theorem \ref{main1}
will be given in Section 3. Finally, in Section 4, we present an open problem on $q$-congruences for further study.

\section{Preliminary Results}\label{pre}

We first recall some basic properties of the harmonic numbers $H_n^{(m)}=\sum_{k=1}^n1/k^m$ and the Euler polynomials.

\begin{lemma}[Lehmer \cite{Lehmer}]\label{lehmercon}
For any prime $p>3$,
$$
H_{p-1}\eq0\pmod{p^2}\ \ \t{and}\ \ H_{p-1}^{(2)}\eq H_{(p-1)/2}^{(2)}\eq0\pmod{p}.
$$
\end{lemma}

\begin{lemma}[\hspace{-0.1mm}\cite{MOS}]\label{Euler}
Let $x\in\C$ and $n,m\in\N$. Then
\begin{align*}
E_{2n}(0)&=E_{2n}(1)=0,\\
E_n(1-x)&=(-1)^nE_n(x),\\
\sum_{k=1}^n(-1)^kk^m&=\f{(-1)^n}{2}\l(E_m(n+1)+(-1)^nE_m(0)\r).
\end{align*}
\end{lemma}

In order to prove Theorem \ref{main1}, we need to establish another six lemmas.
For the sake of convenience, we shall always assume that $\alpha+\langle -\alpha\rangle_p=pt$ from now on.

\begin{lemma}\label{a=0}
If $\alpha\eq0\pmod{p}$, then Theorem \ref{main1} holds.
\end{lemma}

\begin{proof}
It is easy to see that
\begin{align*}
\sum_{k=0}^{p-1}(-1)^k(2k+pt)\f{(pt)_k^3}{(1)_k^3}=&\ pt+\sum_{k=1}^{p-1}(-1)^k(2k+pt)\f{(pt)_k^3}{(1)_k^3}\\
\eq&\ pt+2p^3t^3\sum_{k=1}^{p-1}(-1)^{k}k\f{(1+pt)_{k-1}^3}{(1)_{k}^3}\\
\eq&\ pt+2p^3t^3\sum_{k=1}^{p-1}\f{(-1)^k}{k^2}\\
=&\ pt+2p^3t^3\l(\f12H_{(p-1)/2}^{(2)}-H_{p-1}^{(2)}\r)\pmod{p^4}.
\end{align*}
Then, by Lemma \ref{lehmercon},
$$
\sum_{k=0}^{p-1}(-1)^k(2k+pt)\f{(pt)_k^3}{(1)_k^3}\eq pt\pmod{p^4}.
$$
Now, it suffices to show $E_{p-3}(pt)\eq 0\pmod{p}$. In fact, by Lemma \ref{Euler} we have
$$
E_{p-3}(pt)\eq E_{p-3}(0)=0\pmod{p}.
$$
This proves \eqref{main1eq}. Further, the proof of \eqref{main1eq'} is trivial.
\end{proof}

\begin{lemma}\label{wzprod}
For any prime $p>3$ and $\alpha\in\Z_p$ with $\alpha\not\eq0\pmod{p}$, we have the following congruence modulo $p^4$:
$$
\f{(\alpha)_{2p-1}}{(1)_{p-1}^2}\eq\begin{cases}pt,\ &\t{if}\ \ \langle -\alpha\rangle_p=p-1,\\[5pt]
-\dfrac{p^2t(t+1)}{\langle -\alpha\rangle_p+1}\l(1+2pH_{\langle -\alpha\rangle_p}+\dfrac{p(t+2)}{\langle -\alpha\rangle_p+1}\r),\ &\t{otherwise}.
\end{cases}
$$
\end{lemma}

\begin{proof}
If $\langle -\alpha\rangle_p=p-1$, then
\begin{align*}
\f{(\alpha)_{2p-1}}{(1)_{p-1}^2}&=\f{(1+p(t-1))_{2p-1}}{(1)_{p-1}^2}=\f{pt(1+p(t-1))_{p-1}(1+pt)_{p-1}}{(1)_{p-1}^2}\\
&\eq\  pt\l(1+p(t-1)H_{p-1}+\f{p^2(t-1)^2}2(H_{p-1}^2-H_{p-1}^{(2)})\r) \\
&\quad\times\l(1+ptH_{p-1}+\f{p^2t^2}2(H_{p-1}^2-H_{p-1}^{(2)})\r)\\
&\eq\ pt\pmod{p^4},
\end{align*}
where we have used Lemma \ref{lehmercon}.

Below, we assume that $1\ls \langle -\alpha\rangle_p\ls p-2$ and let $a=\langle -\alpha\rangle_p$. Since
$$
H_{p-2-a}=\sum_{j=1}^{p-2-a}\f{1}{j}=\sum_{j=a+2}^{p-1}\f{1}{p-j}\eq -(H_{p-1}-H_{a+1})=H_{a+1}\pmod{p},
$$
we have
\begin{align*}
\f{(\alpha)_{2p-1}}{(1)_{p-1}^2}&=\f{(-a+pt)_{2p-1}}{(1)_{p-1}^2}=\f{p^2t(t+1)(-a+pt)_a(1+pt)_{p-1}(1+p(t+1))_{p-2-a}}{(1)_{p-1}^2}\\
&\eq\  p^2t(t+1)\f{(-a)_a(1)_{p-2-a}(1-ptH_a+p(t+1)H_{p-2-a})}{(1)_{p-1}}\\
&= -p^2t(t+1)\f{(1)_a(1-ptH_a+p(t+1)H_{p-2-a})}{(1-p)_{a+1}}\\
&\eq -\f{p^2t(t+1)(1-ptH_a+p(t+1)H_{p-2-a}+pH_{a+1})}{a+1}\\
&\eq -\f{p^2t(t+1)}{a+1}\l(1+2pH_a+\f{p(t+2)}{a+1}\r)\pmod{p^4}.
\end{align*}
This concludes the proof.
\end{proof}

\begin{lemma}\label{identities}
For any positive integer $n$, we have the following identities:
\begin{align}
\label{Sigma1lem}\sum_{k=1}^n\f{(-1)^k}{k^2\binom{n}{k}}&=H_n^{(2)}+2\sum_{k=1}^n\f{(-1)^k}{k^2},\\
\label{id1}\sum_{k=1}^n\f{(-1)^k}{k}\binom{n}{k}&=-H_n,\\
\label{id2}\sum_{k=1}^n\f{(-1)^k}{k^2}\binom{n}{k}&=-\f12(H_n^{(2)}+H_n^2),\\
\label{id3}\sum_{k=1}^n\f{(-1)^k}{k}\binom{n}{k}H_k&=-H_n^{(2)}.
\end{align}
\end{lemma}

\begin{proof}
We think these identities should be known. For example, the identity \eqref{id1} was recorded in \cite{Gould}.
A $q$-analogue of \eqref{id1} was given by Van Hamme \cite{Hamme2}. The identity \eqref{id2} is a special case of a result due to Dilcher \cite{Dilcher}.
Of course, all of these identities can be automatically found and proved by using the package \verb"Sigma" in Mathematica (cf. \cite{S}).
\end{proof}

\begin{lemma}\label{Sigma1}
For any prime $p$ and $\alpha\in\Z_p$ with $\alpha\not\eq0\pmod{p}$, we have
$$
\f{(\alpha)_p^2}{(1)_{p-1}^2}\sum_{k=1}^{\langle -\alpha\rangle_p}(-1)^k\f{(\alpha)_{p+k-1}}{(1)_{p-k}(\alpha)_k^2}
\eq (-1)^{\langle -\alpha\rangle_p+1}p^3t^3\l(H_{\langle -\alpha\rangle_p}^{(2)}+2\sum_{k=1}^{\langle -\alpha\rangle_p}\f{(-1)^k}{k^2}\r)\pmod{p^4}.
$$
\end{lemma}

\begin{proof}Similarly as before, let $a=\langle -\alpha\rangle_p$.
It is easy to see that
\begin{align}
\f{(\alpha)_p^3}{(1)_{p-1}^3}&=\f{(pt)^3(\alpha)_a^3(1+pt)_{p-1-a}^3}{(1)_{p-1}^3}\label{alphap3}\\
&\eq\f{(pt)^3(-a)_a^3(1)_{p-1-a}^3}{(1)_{p-1}^3}\notag\\
&=\f{(pt)^3(-1)^a(-a)_a^3}{(1-p)_{a}^3}\notag\\
&\eq p^3t^3\pmod{p^4}.\notag
\end{align}
For any positive integer $k$, there hold
$$
(\alpha)_{p+k-1}=(\alpha)_p(\alpha+p)_{k-1},
$$
and
$$
(1)_{p-k}=(-1)^{k-1}\f{(1)_{p-1}}{(1-p)_{k-1}}.
$$
Then, by \eqref{Sigma1lem} we obtain
\begin{align*}
\f{(\alpha)_p^2}{(1)_{p-1}^2}\sum_{k=1}^a(-1)^k\f{(\alpha)_{p+k-1}}{(1)_{p-k}(\alpha)_k^2}
&=-\f{(\alpha)_p^3}{(1)_{p-1}^3}\sum_{k=1}^{a}\f{(\alpha+p)_{k-1}(1-p)_{k-1}}{(\alpha)_k^2}\\
&\eq-p^3t^3\sum_{k=1}^a\f{(1)_{k-1}}{(-a+k-1)^2(-a)_{k-1}}\\
&=-p^3t^3\sum_{k=0}^{a-1}\f{(-1)^k}{(-a+k)^2\binom{a}{k}}\\
&= (-1)^{a+1}p^3t^3\sum_{k=1}^{a}\f{(-1)^k}{k^2\binom{a}{k}}\\
&= (-1)^{a+1}p^3t^3\l(H_a^{(2)}+2\sum_{k=1}^a\f{(-1)^k}{k^2}\r)\pmod{p^2},
\end{align*}
as desired.
\end{proof}

\begin{lemma}\label{Prod}
For any prime $p>3$ and $\alpha\in\Z_p$, we have
\begin{equation}\label{Prodeq}
\f{(\alpha)_p^2(\alpha)_{p+a}}{(1)_{p-1}^2(1)_{p-a-1}(\alpha)_{a+1}^2}\eq pt+p^2t(t+1)H_a+\f{p^3t(t+1)^2}{2}H_a^2+\f{p^3t(t^2+4t+1)}{2}H_a^{(2)}\pmod{p^4},
\end{equation}
where $a=\langle -\alpha\rangle_p$.
\end{lemma}

\begin{proof}
As in \eqref{alphap3}, the left-hand side of \eqref{Prodeq} is equal to
\begin{align*}
&(-1)^a\f{(\alpha)_p^3(\alpha+p)_a(1-p)_a}{(1)_{p-1}^3(pt)^2(\alpha)_a^2}\\
&\quad=(-1)^apt\f{(1+pt)_{p-1-a}^3(\alpha+p)_a(1-p)_a}{(1)_{p-1}}\\
&\quad=pt\f{(1+pt)_{p-1}^3(1-pt)_a(1-p)_a}{(1)_{p-1}^3(1-p(t+1))_a^2}.
\end{align*}
It is routine to check that for any  $u\in p\Z_p$ and $k\in\N$,
$$
(1+u)_k\eq(1)_k\l(1+uH_k+u^2\sum_{1\ls i<j\ls k}\f{1}{ij}\r)=(1)_k\l(1+uH_k+\f{u^2}{2}(H_k^2-H_k^{(2)})\r)\pmod{p^3}.
$$
Thus, in view of Lemma \ref{lehmercon}, we arrive at
$$
\f{(1+pt)_{p-1}^3}{(1)_{p-1}^3}\eq\l(1+ptH_{p-1}+\f{p^2t^2}2H_{(p-1)/2}\r)^3\eq1\pmod{p^3}
$$
and
\begin{align*}
&\f{(1-pt)_a(1-p)_a}{(1-p(t+1))_a^2}\\
&\quad\eq\f{\l(1-ptH_a+\f{p^2t^2}2(H_a^2-H_a^{(2)})\r)\l(1-pH_a+\f{p^2}2(H_a^2-H_a^{(2)})\r)}{\l(1-p(t+1)H_a+\f{p^2(t+1)^2}2(H_a^2-H_a^{(2)})\r)^2}\\
&\quad\eq1+p(t+1)H_a+\f{p^2(t+1)^2}{2}H_a^2+\f{p^2(t^2+4t+1)}{2}H_a^{(2)}\pmod{p^3}.
\end{align*}
Combining the above congruences, we are led to \eqref{Prodeq}.
\end{proof}

\begin{lemma}\label{sigma}
For any prime $p>3$ and $\alpha\in\Z_p$ with $a=\langle -\alpha\rangle_p\ls p-2$, we have
\begin{align}\label{sigmaeq}
&\f{(\alpha)_p^2}{(1)_{p-1}^2}\sum_{k=a+2}^{p-1}(-1)^{k}\f{(\alpha)_{p+k-1}}{(1)_{p-k}(\alpha)_k^2}\notag\\
&\quad\eq (-1)^ap^2t(t+1)\l(H_a-\f{(-1)^a}{a+1}\r)+(-1)^ap^3t(t+1)\notag\\
&\qquad\times\l(\f{t+1}{2}H_a^2+\f{3t+1}{2}H_a^{(2)}-\f{(-1)^a2}{a+1}H_a-\f{(-1)^a(t+2)}{(a+1)^2}\r)\pmod{p^4}.
\end{align}
\end{lemma}

\begin{proof}
If $a=p-2$, then the left-hand side of \eqref{sigmaeq} is equal to $0$. Meanwhile, modulo $p^4$, the right-hand side of \eqref{sigmaeq} is congruent to
\begin{align*}
&-p^2t(t+1)H_{p-1}-p^3t(t+1)\Bigg(\f{t+1}{2}\l(H_{p-1}-\f{1}{p-1}\r)^2+\f{3t+1}{2}\l(H_{p-1}^{(2)}-\f{1}{(p-1)^2}\r)\\
&+\f{2}{p-1}\l(H_{p-1}-\f{1}{p-1}\r)+\f{t+2}{(p-1)^2}\Bigg)\\
&\quad\eq-\f{p^3t(t+1)^2}2+\f{p^3t(t+1)(3t+1)}{2}+2p^3t(t+1)-p^3t(t+1)(t+2)\\
&\quad\eq 0,
\end{align*}
where we have applied Lemma \ref{lehmercon}.

In what follows, we suppose that $a\ls p-3$. It is easy to see that
\begin{align}\label{sigmakey}
&\f{(\alpha)_p^2}{(1)_{p-1}^2}\sum_{k=a+2}^{p-1}(-1)^{k}\f{(\alpha)_{p+k-1}}{(1)_{p-k}(\alpha)_k^2}\notag\\
&\quad=-\f{(\alpha)_p^3(\alpha+p)_a(\alpha+a+p)}{(1)_{p-1}^3(\alpha+a)^2(\alpha)_a^2}\sum_{k=a+2}^{p-1}\f{(1+\alpha+a+p)_{k-a-2}(1-p)_{k-1}}{(1+\alpha+a)_{k-a-1}^2}\notag\\
&\quad=-\f{(\alpha)_p^3(\alpha+p)_a(\alpha+a+p)}{(1)_{p-1}^3(\alpha+a)^2(\alpha)_a^2}\sum_{k=1}^{p-a-2}\f{(1+\alpha+a+p)_{k-1}(1-p)_{a+k}}{(1+\alpha+a)_{k}^2}.
\end{align}
Furthermore,
\begin{align}\label{sigmakeyprod}
&-\f{(\alpha)_p^3(\alpha+p)_a(\alpha+a+p)}{(1)_{p-1}^3(\alpha+a)^2(\alpha)_a^2}\notag\\
&\quad=-p^2t(t+1)(-1)^a\f{(-a+pt)_a(1+pt)_{p-1}^3(-a+p(t+1))_a(1)_a}{(1)_{p-1}^3(1-p(t+1))_a^3}\notag\\
&\quad\eq-p^2t(t+1)(-1)^a\f{(-a)_a^2}{(1)_a^2}\l(1+3ptH_{p-1}-p(2t+1)H_a+3p(t+1)H_a\r)\notag\\
&\quad\eq(-1)^{a+1}p^2t(t+1)(1+p(t+2)H_a)\pmod{p^4}.
\end{align}
In addition,
\begin{align}\label{sigmakeysumpart1}
&\sum_{k=1}^{p-a-2}\f{(1+\alpha+a+p)_{k-1}(1-p)_{a+k}}{(1+\alpha+a)_{k}^2}\notag\\
&\quad=\sum_{k=1}^{p-a-2}\f{(1+p(t+1))_{k-1}(1-p)_{a+k}}{(1+pt)_{k}^2}\notag\\
&\quad\eq  (1)_a\sum_{k=1}^{p-a-2}\f{(a+1)_{k}}{k(1)_k}\l(1+p(t+1)H_{k-1}-pH_{a+k}-2p tH_k\r)\notag\\
&\quad\eq   \sum_{k=1}^{p-a-2}\f{(-1)^k\binom{p-a-1}{k}}{k}\l(1+p(t+1)H_k-\f{p(t+1)}{k}-pH_a-2p tH_k\r)\notag\\
&\quad\eq  \sum_{k=1}^{p-a-1}\f{(-1)^k\binom{p-a-1}{k}}{k}\l(1+p(1-t)H_k-\f{p(t+1)}{k}-pH_a\r)\notag\\
&\qquad+(-1)^a\l(\f{1}{a+1}+\f{p}{(a+1)^2}\r)(1+p(t+1)H_{a+1}-pH_a-2p tH_a)\pmod{p^2}.
\end{align}
With the help of Lemma \ref{identities}, we conclude that
\begin{align}\label{sigmakeysumpart2}
&\sum_{k=1}^{p-a-1}\f{(-1)^k\binom{p-a-1}{k}}{k}\l(1+p(1-t)H_k-\f{p(t+1)}{k}-pH_a\r)\notag\\
&\quad=-H_{p-1-a}+pH_aH_{p-1-a}-p(1-t)H_{p-1-a}^{(2)}+\f{p(t+1)}{2}(H_{p-1-a}^{(2)}+H_{p-1-a}^2)\notag\\
&\quad\eq-H_a-pH_a^{(2)}+pH_a^2+p(1-t)H_a^{(2)}+\f{p(t+1)}{2}(H_a^2-H_a^{(2)})\pmod{p^2}.
\end{align}
Combining \eqref{sigmakey}--\eqref{sigmakeysumpart2}, we finish the proof.
\end{proof}

\section{Proof of Theorem \ref{main1}}
As in Section \ref{pre}, we write  $a=\langle -\alpha\rangle_p$ and $\alpha+a=pt$. For non-negative integers $n$ and $k$, we define
$$
F(n,k)=(-1)^{n+k}\f{(2n+\alpha)(\alpha)_n^2(\alpha)_{n+k}}{(1)_n^2(1)_{n-k}(\alpha)_k^2}
$$
and
$$
G(n,k)=(-1)^{n+k}\f{(\alpha)_n^2(\alpha)_{n+k-1}}{(1)_{n-1}^2(1)_{n-k}(\alpha)_k^2},
$$
where we assume that $1/(1)_m=0$ if $m<0$. Then we can easily verify that
\begin{equation}\label{WZpair}
F(n,k-1)-F(n,k)=G(n+1,k)-G(n,k).
\end{equation}

\medskip

\noindent{\it Proof of \eqref{main1eq}}.  Summing both sides of \eqref{WZpair} over $n$ from $0$ to $p-1$, we get
$$
\sum_{n=0}^{p-1}F(n,k-1)-\sum_{n=0}^{p-1}F(n,k)=G(p,k)-G(0,k)=G(p,k).
$$
Further, summing the above identity over $k$ from $1$ to $p-1$, we obtain
$$
\sum_{n=0}^{p-1}F(n,0)=F(p-1,p-1)+\sum_{k=1}^{p-1}G(p,k).
$$
Namely, we have
\begin{equation}\label{wzkey}
\sum_{k=0}^{p-1}(-1)^k(2k+\alpha)\f{(\alpha)_k^3}{(1)_k^3}=\f{(\alpha)_{2p-1}}{(1)_{p-1}^2}-\f{(\alpha)_p^2}{(1)_{p-1}^2}\sum_{k=1}^{p-1}(-1)^k\f{(\alpha)_{p+k-1}}{(1)_{p-k}(\alpha)_k^2}.
\end{equation}

We consider three cases.
If $a=0$, then by Lemma \ref{a=0}, the congruence \eqref{main1eq} holds in this case.
If $a=p-1$, then by \eqref{wzkey},
$$
\sum_{k=0}^{p-1}(-1)^k(2k+\alpha)\f{(\alpha)_k^3}{(1)_k^3}=\f{(\alpha)_{2p-1}}{(1)_{p-1}^2}-\f{(\alpha)_p^2}{(1)_{p-1}^2}\sum_{k=1}^{a}(-1)^k\f{(\alpha)_{p+k-1}}{(1)_{p-k}(\alpha)_k^2}.
$$
With the help of Lemmas \ref{wzprod} and \ref{Sigma1}, we obtain
\begin{align*}
\sum_{k=0}^{p-1}(-1)^k(2k+\alpha)\f{(\alpha)_k^3}{(1)_k^3}&\eq pt+p^3t^3\l(H_{p-1}+2\sum_{k=1}^{p-1}\f{(-1)^k}{k^2}\r)\\
&=pt+p^3t^3\l(H_{p-1}+H_{(p-1)/2}^{(2)}-2H_{p-1}^{(2)}\r)\\
&\eq pt\pmod{p^4},
\end{align*}
where we have utilized \eqref{lehmercon} in the last step. Then the desired result follows from the fact that
$$
E_{p-3}(\alpha)\eq E_{p-3}(1)=E_{p-3}(0)=0\pmod{p}.
$$
If $1\ls a\ls p-2$, then, in light of \eqref{wzkey}, the left-hand side of \eqref{main1eq} is equal to
\begin{align}\label{case3key}
&\f{(\alpha)_{2p-1}}{(1)_{p-1}^2}-\f{(\alpha)_p^2}{(1)_{p-1}^2}\sum_{k=1}^{a}(-1)^k\f{(\alpha)_{p+k-1}}{(1)_{p-k}(\alpha)_k^2}-(-1)^a\f{(\alpha)_p^2(\alpha)_{p+a}}{(1)_{p-1}^2(1)_{p-a-1}(\alpha)_{a+1}^2}\notag\\
&\qquad-\f{(\alpha)_p^2}{(1)_{p-1}^2}\sum_{k=a+2}^{p-1}(-1)^k\f{(\alpha)_{p+k-1}}{(1)_{p-k}(\alpha)_k^2}.
\end{align}
By Lemmas \ref{wzprod}, \ref{Sigma1}--\ref{sigma}, modulo $p^4$, \eqref{case3key} is congruent to
\begin{align*}
&-\f{p^2t(t+1)}{a+1}\l(1+2pH_a+\f{p(t+2)}{a+1}\r)+(-1)^ap^3t^3\l(H_a^{(2)}+2\sum_{k=1}^a\f{(-1)^k}{k^2}\r)\\
& +(-1)^a\l(pt+p^2t(t+1)H_a+\f{p^3t(t+1)^2}{2}H_a^2+\f{p^3t(t^2+4t+1)}{2}H_a^{(2)}\r)\\
& -(-1)^ap^2t(t+1)\l(H_a-\f{(-1)^a}{a+1}\r)-(-1)^ap^3t(t+1)\\
&\times\l(\f{t+1}{2}H_a^2+\f{3t+1}{2}H_a^{(2)}-\f{(-1)^a2}{a+1}H_a-\f{(-1)^a(t+2)}{(a+1)^2}\r)\\
&\quad=(-1)^apt+(-1)^a2p^3t^3\sum_{k=1}^a\f{(-1)^k}{k^2}.
\end{align*}
In view of Lemma \ref{Euler}, we deduce that
\begin{align*}
\sum_{k=1}^a\f{(-1)^k}{k^2}&\eq \sum_{k=1}^a(-1)^k k^{p-3}\\
&=\f{(-1)^a}{2}(E_{p-3}(a+1)+(-1)^aE_{p-3}(0))\\
&\eq \f{(-1)^a}{2}E_{p-3}(\alpha)\pmod{p}.
\end{align*}
This proves \eqref{main1eq}.\qed

\medskip

\noindent{\it Proof of \eqref{main1eq'}}. It suffices to show that
\begin{equation}\label{main1eq'key}
\sum_{k=a+1}^{p-1}(-1)^k(2k+\alpha)\f{(\alpha)_k^3}{(1)_k^3}\eq0\pmod{p^4}
\end{equation}
for $a\ls p-2$. In fact,
\begin{align*}
&\sum_{k=a+1}^{p-1}(-1)^k(2k-a+pt)\f{(-a+pt)_k^3}{(1)_k^3}\\
&\quad=(-1)^{a+1}\sum_{k=0}^{p-a-2}(-1)^k(2k+a+2+pt)\f{(-a+pt)_{k+a+1}^3}{(1)_{k+a+1}^3}\\
&\quad=\f{(-1)^{a+1}(-a+pt)_{a+1}^3}{(1)_{a+1}^3}\sum_{k=0}^{p-a-2}(-1)^k(2k+a+2+pt)\f{(1+p(t+1))_{k}^3}{(a+2)_{k}^3}.
\end{align*}
Since $(-a+pt)_{a+1}^3/(1)_{a+1}^3\eq0\pmod{p^3}$ and
\begin{align*}
&\sum_{k=0}^{p-a-2}(-1)^k(2k+a+2+pt)\f{(1+p(t+1))_{k}^3}{(a+2)_{k}^3}\\
&\quad\eq\sum_{k=0}^{p-a-2}\f{2k+a+2}{\binom{p-a-2}{k}^3}\\
&\quad=\sum_{k=0}^{p-a-2}\f{2(p-a-2-k)+a+2}{\binom{p-a-2}{k}^3}\\
&\quad\eq-\sum_{k=0}^{p-a-2}\f{2k+a+2}{\binom{p-a-2}{k}^3}\\
&\quad\eq 0\pmod{p},
\end{align*}
as desired. \qed

\section{Concluding remarks and an open problem}
For any odd prime $p$, let $\big(\frac{\cdot}{p}\big)$ be the Legendre symbol modulo $p$.
Recently, using a WZ pair in \cite{CXH}, Mao \cite{Mao} proved the following two supercongruences: for any odd prime $p$,
\begin{align*}
\sum_{k=0}^{(p-1)/2}(-1)^k (6k+1)\frac{(\frac12)_k^3}{k!^3 8^k}
\equiv p\left(\frac{-2}{p}\right)+\frac{p^3}{4}\left(\frac{2}{p}\right)E_{p-3},
\end{align*}
which was originally conjectured by Sun in \cite[Conjecture 5.1]{Sun0}, and
\begin{align}
\sum_{k=0}^{p-1}(-1)^k (6k+1)\frac{(\frac12)_k^3}{k!^3 8^k}
\equiv p\left(\frac{-2}{p}\right)+\frac{p^3}{16}E_{p-3}(\tfrac14),  \label{eq:sun-mao}
\end{align}
which was the $n=1$ case of a conjecture of Sun \cite[(2.16)]{Sun2}.

It is clear that, for any prime $p\equiv 1\pmod{4}$, we have $\left(\frac{-2}{p}\right)=(-1)^{(p-1)/4}$. Combining the two supercongruences
\eqref{eq:f2-mod4} and \eqref{eq:sun-mao}, we immediately obtain the following result: for any prime $p\equiv 1\pmod{4}$,
\begin{align}
\sum_{k=0}^{p-1}(-1)^k (6k+1)\frac{(\frac12)_k^3}{k!^3 8^k}
\equiv \sum_{k=0}^{p-1}(-1)^k(8k+1)\frac{(\frac{1}{4})_k^3}{k!^3} \pmod{p^4}.  \label{eq:equiv}
\end{align}

Let $[n]=(1-q^n)/(1-q)$ be the $q$-integer, $(x;q)_n=(1-x)(1-xq)\cdots (1-xq^{n-1})$ ($n\geqslant 0$) the $q$-shifted factorial,
and let the $n$-th cyclotomic polynomials be given by
$$
\Phi_n(q)=\prod_{\substack{1\leqslant k\leqslant n\\ \gcd(k,n)=1}}(q-\zeta^k),
$$
where $\zeta$ is an $n$-th primitive root of unity. We believe that the following $q$-analogue of \eqref{eq:equiv} should be true.

\begin{conjecture}\label{conj:first}
Let $n\equiv 1\pmod{4}$ be a positive integer. Then
\begin{align}
\sum_{k=0}^{n-1}(-1)^k[6k+1]\frac{(q;q^2)_k^3}{(q^4;q^4)_k^3}q^{3k^2}
\equiv \sum_{k=0}^{n-1}(-1)^k[8k+1]\frac{(q;q^4)_k^3}{(q^4;q^4)_k^3}q^{2k^2+k} \pmod{[n]\Phi_n(q)^3}. \label{eq:q-equvi}
\end{align}
\end{conjecture}

Note that the first author and Zudilin \cite[Theorem 4.4]{GuoZu} proved that, for any positive odd integer $n$,
$$
\sum_{k=0}^{n-1}(-1)^k[6k+1]\frac{(q;q^2)_k^3}{(q^4;q^4)_k^3}q^{3k^2}
\equiv (-q)^{(n-1)(n-3)/8}[n] \pmod{[n]\Phi_n(q)^2},
$$
and they \cite[Theorem 4.9]{GuoZu} also showed that, for positive integers $n\equiv 1\pmod{4}$,
$$
\sum_{k=0}^{n-1}(-1)^k[8k+1]\frac{(q;q^4)_k^3}{(q^4;q^4)_k^3}q^{2k^2+k}
\equiv (-q)^{(n-1)(n-3)/8}[n] \pmod{[n]\Phi_n(q)^2}.
$$
This means that the $q$-congruence \eqref{eq:q-equvi} holds modulo $[n]\Phi_n(q)^2$. Nevertheless, Conjecture \ref{conj:first}
seems still rather challenging. For a similar but more difficult conjecture, we refer the reader to \cite[Conjecture 6.4]{Guo-aam}.

\end{document}